\documentclass[pdflatex,sn-apa,11pt]{sn-jnl}% APA Reference Style 
\usepackage[shortlabels]{enumitem}
\usepackage{graphicx}%
\usepackage{bm} % my addition on Dec 11
\usepackage{multirow}%
\usepackage{amsmath,amssymb,amsfonts}%
\usepackage{amsthm}%
\usepackage{mathrsfs}%
\usepackage[title]{appendix}%
\usepackage{xcolor}%
\usepackage{textcomp}%
\usepackage{manyfoot}%
\usepackage{booktabs}%
\usepackage{algorithm}%
\usepackage{algorithmicx}%
\usepackage{algpseudocode}%
\usepackage{listings}%
\newcommand{\possessivecite}[1]{\citeauthor{#1}'s (\citeyear{#1})}

\theoremstyle{thmstyleone}%
\newtheorem{theorem}{Theorem}%  meant for continuous numbers

\newtheorem{proposition}{Proposition}% to get separate numbers for theorem and proposition etc.

\theoremstyle{thmstyletwo}%
\newtheorem{example}{Example}%
\newtheorem{remark}{Remark}%
\newtheorem{lemma}{Lemma}
\newtheorem{corollary}{Corollary}

\theoremstyle{thmstylethree}%

\begin{document}

\title[Multiple objective linear programming over the probability simplex]{Multiple objective linear programming over the probability simplex}

%%=============================================================%%
%% GivenName	-> \fnm{Joergen W.}
%% Particle	-> \spfx{van der} -> surname prefix
%% FamilyName	-> \sur{Ploeg}
%% Suffix	-> \sfx{IV}
%% \author*[1,2]{\fnm{Joergen W.} \spfx{van der} \sur{Ploeg} 
%%  \sfx{IV}}\email{iauthor@gmail.com}
%%=============================================================%%

%\author[]{Anas Mifrani*\vspace{-1em}}
%\date{\today}
%\address{*: Toulouse Mathematics Institute, University of Toulouse, France}
%\email{anas.mifrani@math.univ-toulouse.fr}
%\maketitle

%\let\thefootnote\relax
%\footnotetext{*: Toulouse Mathematics Institute, University of Toulouse, UT3, F-31062 Toulouse Cedex 9, France.} %%%%%%%%%%
%\footnotetext{Email address: anas.mifrani@math.univ-toulouse.fr.}
%\footnotetext{ORCID: 0009-0005-1373-9028.}
%\footnotetext{Declarations of interest: none.}

%\author[]{\fnm{Anas} \sur{Mifrani}}%\email{anas.mifrani@math.univ-toulouse.fr}

%\affil[]{\orgdiv{Toulouse Mathematics Institute}, \orgname{University of Toulouse}, \orgaddress{\street{118 Route de Narbonne}, \city{Toulouse}, \postcode{F-31062}, \state{France}}}

\author[]{\fnm{Anas} \sur{Mifrani}\footnote{Toulouse Mathematics Institute, University of Toulouse, F-31062 Toulouse Cedex 9, France. Email address: anas.mifrani@math.univ-toulouse.fr.\\ ORCID: 0009-0005-1373-9028.}}

%%==================================%%
%% Sample for unstructured abstract %%
%%==================================%%

\abstract{
This paper considers the problem of maximizing multiple linear functions over the probability simplex. A classification of feasible points is indicated. A necessary and sufficient condition for a member of each class to be an efficient solution is stated. This characterization yields a computational procedure for ascertaining whether a feasible point is efficient. The procedure does not require that candidates for efficiency be extreme points. An illustration of the procedure is offered.
%The problem originates in a theorem of Po-Lung Yu that asserts the equality for a specific choice of sets and a specific choice of relation. 
%We generalize Po-Lung Yu's theorem that the efficient points in a sum set, for a particular choice of sets, reduce to the efficient points in one of the summands. In doing so we also correct Yu's proof.
}

\keywords{Multiple objective linear program, Efficient solution, Multi-objective optimization, Vector maximization, Probability simplex.}

%%\pacs[JEL Classification]{D8, H51}

%%\pacs[MSC Classification]{35A01, 65L10, 65L12, 65L20, 65L70}

\maketitle

\section{Introduction}\label{sec1}

Consider $k \geq 2$ vectors $c_1, ..., c_k$ in $\mathbb{R}^{n}$, where $n \geq 2$. Let $X$ denote the set of all points $x \in \mathbb{R}^{n}$ that satisfy $\sum_{i = 1}^{n}x_j = 1$ and $x_j \geq 0$ for all $j = 1, ..., n$. We shall concern ourselves with the problem of simultaneously maximizing the functions $c_i^{T}x$, $i = 1, ..., k$, over $X$. This \textit{multiple objective programming problem} may be expressed as
\begin{equation}
\tag{P}
\label{P}
\begin{aligned}
\operatorname{VMAX:} \quad & Cx = (c_1^{T}x, ..., c_k^{T}x)^T, \textrm{ subject to } x \in X,\\
\end{aligned}
\end{equation}
where $C$ denotes the $k \times n$ matrix whose $i$th row, $i = 1, ..., k$, contains the vector $c_i$.

Each point in $X$ may be interpreted as a probability distribution. If we let $c_{i, j}$ denote the $j$th component of $c_i$, for each $i = 1, ..., k$, then $c_i^{T}x$ represents the expectation of a random variable that takes values in $\{c_{i, 1}, ..., c_{i, n}\}$ assuming each $c_{i, j}$ occurs with probability $x_j$. Problem (\ref{P}) can therefore be viewed as the problem of finding the probability distributions that ensure the largest possible expectation for a multivariate random variable with realizations in $\{c_1, ..., c_k\}$. 

To avoid trivialities, we assume that no single point exists at which all the functions of this problem achieve their maximum values\footnote{This assumption merely serves as motivation for the paper; it has no implications for the validity of the results developed in it. If there exists a point at which all the functions are maximized, then Problem (\ref{P}) reduces, in effect, to an ordinary mathematical programming problem involving any one of the criteria. At that point, our findings, whilst still applicable, may lose some of their originality as they may coincide with familiar facts of standard mathematical programming.}. The concept of efficiency has figured prominently in the formal analysis and solution of such problems. %\footnote{See, for instance, \cite{soland1979multicriteria} for theoretical justifications for the widespread adoption of this concept.}
We call a point $x^{\circ} \in X$ an \textit{efficient solution} of Problem (\ref{P}) when there is no $x \in X$ such that $c_i^{T}x \geq c_i^{T}x^{\circ}$ for all $i = 1, ..., k$ with at least one strict inequality. Points which are not efficient solutions are said to be \textit{dominated}. Each of the functions $c_i^{T}x$ is called a \textit{criterion} of Problem (\ref{P}). Let $X_E$ denote the set of all efficient solutions of Problem (\ref{P}). %Stated differently, $x^{\circ} \in X$ is an efficient solution of (\ref{P}) when no other solution achieves  %Equivalently, $x^{\circ}$ is an efficient solution of (\ref{P}) if $c_i^{T}x > c_i^{T}x^{\circ}$ for some $x \in X$ and some $i$ implies the existence of a $j$ such that $c_j^{T}x^{\circ} > c_j^{T}x$.

%The usefulness of Problem (\ref{P}) is perhaps illustrated by the following simple, but realistic, situation. Suppose a decision maker has available a set of $n$ alternatives together with a set of $k$ criteria for assessing the alternatives. There is a certain measure of conflict among the criteria in the assessments given, so that the best alternatives with respect to one criterion may also be the poorest with respect to other criteria. In these circumstances, it is advisable that the decision maker settle for efficient alternatives. One may also allow for randomization in the choice of alternatives.

Since its criteria are linear and $X$ is a convex polyhedral set, Problem (\ref{P}) belongs to the class of multiple objective linear programs. These are some of the simplest and most extensively investigated of multiple objective programs. Already 26 years ago, \cite{benson1998outer} reported that they had been ``studied in literally hundreds of articles, chapters in books, and books". Some recurring topics of research have been the existence of efficient solutions \citep{benson1978existence, ecker1975finding}, properties of the efficient solution set \citep{steuer1986multiple, yu2013multiple, benson1995geometrical}, the relation to ordinary linear programming \citep{GEOFFRION1968618, evans1973revised}, and methods for generating the efficient solution set \citep{isermann1977enumeration, armand1991determination, sayin1996algorithm} or merely that portion of it which lies in the set of extreme points of the constraint polyhedron \citep{evans1973revised, ecker1978finding}. %There has also been no shortage of applications for these programs, real or imagined, in such areas as nutrition planning \citep{benson1987bicriteria}, farm planning \citep{annetts2002multiple}, and portfolio selection \citep{ogryczak2000multiple}.

Researchers have often found it fruitful when working with special instances of mathematical programming problems (such as Problem (\ref{P})) to try to analyze them as thoroughly as possible before considering applying the general theory. The hope is that, as a consequence of such an analysis, properties will emerge which suggest specialized solution procedures that are better suited for the problem at hand than the general procedures, either because unnecessary computations built into the latter are avoided, or simply because of superior ease of implementation. As an illustration of this approach, take \possessivecite{benson1979vector} treatment of a multiple objective program involving two concave functions defined on a convex set. With the help of an alternate definition of efficiency, he shows that the efficient solutions solve a set of ordinary programming problems in which one criterion is to be maximized while the other is bounded below. %discovers a necessary and sufficient condition for a point to be an efficient solution to the multiple objective program. 
This characterization leads to a parametric procedure for finding all efficient solutions. Contrary to various procedures devised for more general problems, Benson's ``does not require continual suboptimizations or examinations of the feasibility of systems of equations" \citep[p. 4]{benson1979vector}.

Our aim in this paper is to pursue a similar approach in relation to Problem (\ref{P}), and to consider the results. For reasons that will be clear from the development, we will classify the points in $X$ into three groups, each of which having a certain geometrical interpretation. %the \textit{pure}, the \textit{semirandomized}, and the \textit{randomized}.
For each group, we shall state a necessary and sufficient condition for a member to be an efficient solution of Problem (\ref{P}). For one group in particular, we shall demonstrate that the existence of an efficient solution among its members implies that every point in $X$ is an efficient solution. A computational procedure will follow that indicates whether or not a given $x^{\circ} \in X$ is an efficient solution. If $x^{\circ}$ is efficient, the procedure will, in certain cases, reveal an entire region of $X$ consisting only of efficient solutions. Unlike the standard efficiency tests for multiple objective linear programming, the tests employed by this procedure do not require that the points under consideration represent vertices of $X$. These results are developed in Section \ref{sec2}, then illustrated in Section \ref{sec3}. Section \ref{sec4} summarizes the paper.   
%collects conclusions on the matter
%seldom

%Researchers have, in the past, looked at special cases of multiple objective programming problems with a view to developing specialized solution procedures that are 

%shed some light on

\section{Characterization of efficient solutions}\label{sec2}

We preface the development with some additional notation. Let $J = \{1, ..., n\}$ and $I = \{1, ..., k\}$. For each $j \in J$, let $\bm{e_j}$ signify the $n$-dimensional vector with a $1$ in component $j$ and zeros elsewhere. For any $d \in \mathbb{R}^{n}$, let $d_{\max}$ denote the value of the largest component in $d$, and $J^{*}(d)$ denote the set $\{j \in J: d_j = d^{*}\}$. Finally, write $\Lambda^{>} = \{\lambda \in \mathbb{R}^{k}: \lambda_{i} > 0,\ \forall i \in I\}$.

A basic question in multiple objective programming is whether and when efficient solutions exist. Even linear programs need not have efficient solutions; see, for example, Theorem 3.2(iii) of \cite{evans1973revised}. Fortunately, the issue does not arise for Problem (\ref{P}), as is revealed by the following proposition. 

%for each $j \in J$, denote by $J_j$ the set $J \setminus \{j\}$. 

%For each $j \in J$, let $\bm{e_j}$ signify the $n$-dimensional canonical basis vector with a $1$ in the $j$th component.   

\begin{proposition}
\label{prop0}
The efficient solution set $X_E$ is nonempty.
\end{proposition}

\begin{proof}
According to Corollary 4.6 of \cite{yu1974cone} as it was reported in \cite{benson1978existence}, a sufficient condition for $X_E$ to be nonempty is that the set $\{Cx: x \in X\}$ be closed and bounded. That this is indeed the situation (in, say, sup-norm topology) is readily verifiable.
\end{proof}
%$J^{*}(d) = \{j \in J: d_j \geqq d_m \text{ for all } m \in I\}$ be the set of all indices corresponding to the largest component in $d$.  

For reasons that will soon become apparent, we will have some interest in ordinary linear programs over $X$. Lemma \ref{lem1} gives a complete specification of the optimal solution set for such a program. For maximum generality, we state the lemma with an $n \geq 3$ in mind, then present the corresponding result for $n = 2$ in Remark \ref{rem0}.    

\begin{lemma}
\label{lem1}
Assume that $d \in \mathbb{R}^{n}$. Let $S$ denote the set of all optimal solutions to the ordinary linear program given by
\begin{equation}
\tag{$P_{d}$}
\label{$P_{d}$}
\begin{aligned}
\operatorname{\max} \quad & d^{T}x, \textrm{ subject to } x \in X.\\
\end{aligned}
\end{equation}
Then one and only one of three situations obtains:
\begin{enumerate}
\item[{\rm (i)}] $S = X$;
\item[{\rm (ii)}] there exists a unique $j \in J$ such that $S = \{\bm{e_j}\}$;
\item[{\rm (iii)}] or there exist $p = 2, ..., n-1$ distinct indices $j_1, ..., j_p \in J$ such that $S = \{x_{j_1}\bm{e_{j_1}} + ... + x_{j_p}\bm{e_{j_p}}: \sum_{i = 1}^{p}x_{j_i} = 1,\ x_{j_i} > 0\}$.
\end{enumerate}
Situation (i) arises if and only if $J^{*}(d) = J$; situation (ii) arises if only if $J^{*}(d) = \{j\}$; and situation (iii) arises if and only if $J^{*}(d) = \{j_1, ..., j_p\}$.
%Situation (i) arises if and only if $d_1 = ... = d_n$; situation (ii) if only if the components of $d$ are maximized uniquely at component $i$; and situation (iii) arises if and only if the components of $d$ are maximized simultaneously at $i_1, ..., i_k$.
\end{lemma}
\begin{proof}
Since $d^{T}x$ is a continuous function on $X$, a compact set, the maximum in Problem (\ref{$P_{d}$}) exists, meaning that $S \neq \emptyset$. Let $v$ denote this maximum. The three situations are obviously mutually exclusive. To show that they can actually occur, it suffices to substantiate the second part of the theorem. As a preliminary, notice that $d^{T}x \leq d^{*}$ for every $x \in X$. Consequently, $v \leq d^{*}$.

\textit{Situation (i) arises if $J^{*}(d) = J$.}
Suppose that $J^{*}(d) = J$. Then $d_{j} = d^{*}$ for each $j \in J$, so that for every $x \in X$, $d^{T}x = d^{*}\sum_{j = 1}^{n}x_j = d^{*}$. Thus, from the previous observation, $d^{T}x \geq v$ and therefore $d^{T}x = v$ for every $x \in X$. This implies that every $x \in X$ is an optimal solution to Problem (\ref{$P_{d}$}).

\textit{Situation (ii) arises if $J^{*}(d) = \{j\}$.}
Next, suppose that $J^{*}(d) = \{j\}$ for some $j \in J$. Then $d^{T}\bm{e_j} = d_{j} = d^{*}$, so that $\bm{e_j} \in S$. Let $x \in X$ such that $x \neq \bm{e_j}$. Then there exists $j' \neq j$ such that $x_{j'} > 0$. This, combined with the fact that $d_{j'} < d_{j}$ for each $j' \neq j$, yields %$d^{T}(x - \bm{e_j}) < 0$. 
\begin{equation*}d^{T}x = \sum_{\substack{
    j' = 1 \\
    j' \neq j}}^{n}d_{j'}x_{j'} + d_{j}x_{j} < \sum_{\substack{
    j' = 1 \\
    j' \neq j}}^{n}d_{j}x_{j'} + d_{j}x_{j} = d_{j} = d^{*}.
\end{equation*} 
Thus, $x \notin S$. It follows that $S = \{\bm{e_j}\}$.

To simplify notational problems in relation to situation (iii), let $Y$ denote the set $\{x_{j_1}\bm{e_{j_1}} + ... + x_{j_p}\bm{e_{j_p}}: \sum_{i = 1}^{p}x_{j_i} = 1,\ x_{j_i} > 0\}$, and $J'$ denote the set $\{j_1, ..., j_p\}$. Notice that every $y \in Y$ is feasible for Problem (\ref{$P_{d}$}).

\textit{Situation (iii) arises if $J^{*}(d) = \{j_1, ..., j_p\}$.}
Suppose that $J^{*}(d) = \{j_1, ..., j_p\}$ for some $p = 2, ..., n-1$ indices $j_1, ..., j_p \in J$. Then $d_{j_1} = ... = d_{j_p} = d^{*}$. Moreover, for any $y \in Y$, \begin{equation*}d^{T}y = \sum_{i = 1}^{p}d_{j_i}x_{j_i} = d^{*},\end{equation*}
so that $Y \subseteq S$. Consider now any $x \in X \setminus Y$. As with the previous case, there exists $j' \notin J'$ such that $x_{j'} > 0$. Because $d_{j'} < d^{*}$ for each $j' \notin J'$, \begin{equation*}d^{T}x = \sum_{\substack{
    j' = 1 \\
    j' \notin J'}}^{n}d_{j'}x_{j'} + \sum_{i = 1}^{p}d^{*}x_{j_i} < \sum_{\substack{
    j' = 1 \\
    j' \notin J'}}^{n}d^{*}x_{j'} + \sum_{i = 1}^{p}d^{*}x_{j_i} = d^{*}.\end{equation*}
Therefore, $x \notin S$, whence $S = Y$.

To recapitulate, we have shown that the sufficiency portion of the theorem holds. The necessity portion results from this very fact in conjunction with the observation that the three situations, as well as the three values considered for $J^{*}(d)$, are mutually exclusive.
\end{proof}

\begin{remark}
\label{rem0}
Lemma \ref{lem1} holds for $n = 2$ when situation (iii) is eliminated. 
\end{remark}

\textit{The bulk of the succeeding development presupposes that $n \geq 3$}, but the main results are also applicable when $n = 2$. We return to this matter in Remark \ref{rem3}.

In the introduction to this work, we mentioned that Problem (\ref{P}) is an example of a multiple objective linear program. Evans and Steuer have established that a feasible solution to such a program is efficient if and only if it optimizes a weighted linear combination of the original criteria for some set of positive weights \citep[Corollary 1.4]{evans1973revised}. Their result will be instrumental in characterizing the efficient solutions of Problem (\ref{P}).

\begin{lemma}
\label{lem2}
Let $x^{\circ} \in X$. Then $x^{\circ} \in X_E$ if and only if there exists $\lambda \in \Lambda^{>}$ such that $x^{\circ}$ is an optimal solution of the ordinary linear program
\begin{equation*}
\begin{aligned}
\max \quad & (\lambda^{T}C)x = \sum_{j = 1}^{n}\Biggl(\sum_{i = 1}^{k}\lambda_{i}c_{ij}\Biggr)x_{j}, \textrm{ subject to } x \in X.\\
\end{aligned}
\end{equation*}
\end{lemma}

Put otherwise, $x^{\circ} \in X$ is an efficient solution of Problem (\ref{P}) if and only if there exists some $\lambda \in \Lambda^{>}$ such that $x^{\circ}$ is an optimal solution of Problem ($P_{\lambda^{T}C}$) as defined in Lemma \ref{lem1}. Given a $\lambda \in \mathbb{R}^{k}$, let $S(\lambda)$ denote the optimal solution set of Problem ($P_{\lambda^{T}C}$). We know from Lemma \ref{lem1} that $S(\lambda)$ is entirely determined by the relative ranking of the coefficients $(\lambda^{T}C)_j$, $j \in J$, in the following way: if the coefficients are uniquely maximized at a component $j \in J$, then $S(\lambda) = \{\bm{e_j}\}$; alternatively, if this maximum is attained not in one but $2 \leq p < n$ components $j_1, ..., j_p \in J$, then $S(\lambda) = \{x_{j_1}\bm{e_{j_1}} + ... + x_{j_p}\bm{e_{j_p}}: \sum_{i = 1}^{p}x_{j_i} = 1,\ x_{j_i} > 0\}$; and if all coefficients are equal, then $S(\lambda) = X$.   

This naturally leads us to distinguish three classes of points in $X$: those $x$ with $x_j > 0$ for all $j \in J$; those $x$ with $x_{j_i} > 0$ and $\sum_{i = 1}^{p}x_{j_i} = 1$ for only a subset of $p = 2, ..., n-1$ indices $j_1, ..., j_p \in J$; and those $x$ with $x_j = 1$ for a single $j \in J$. We shall refer to the first class of points as the \textit{randomized}, to the second as the \textit{partially randomized}, and to the third as the \textit{deterministic}\footnote{The randomized-deterministic terminology is motivated by the comment in Section \ref{sec1} to the effect that the elements of $X$ define probability distributions over $J$.}. The three classes obviously exhaust all of $X$, and it is clear from their definitions that every member of $X$ belongs to exactly one of them. %he vertices of $X$, for instance, are identical to the deterministic points. %straightforward exercise to show that the deterministic points represent the extreme points of $X$.

%% DO NOT FORGET TO DEMONSTRATE THAT X_E IS NONEMPTY

Consider any $x^{\circ} \in X_E$, and let $\lambda \in \Lambda^{>}$ be a weight factor such that $x^{\circ} \in S(\lambda)$. If $x^{\circ}$ is randomized, then of the three cases that have been discussed, only the third, namely that $S(\lambda) = X$, is possible, which implies that $(\lambda^{T}C)_1 = (\lambda^{T}C)_n$. If $x^{\circ}$ is partially randomized w.r.t.\ some $p = 2, ..., n-1$ components $j_1, ..., j_p$, then the preceding discussion makes it evident that we must have $S(\lambda) = \{x_{j_1}\bm{e_{j_1}} + ... + x_{j_p}\bm{e_{j_p}}: \sum_{i = 1}^{p}x_{j_i} = 1,\ x_{j_i} > 0\}$. In this case, $(\lambda^{T}C)_{j_1} = ... = (\lambda^{T}C)_{j_p}$ and $(\lambda^{T}C)_{j_i} > (\lambda^{T}C)_{j}$ for each $i = 1, ..., p$ and $j \in J \setminus \{j_1, ..., j_p\}$. Finally, if $x^{\circ}$ is deterministic with respect to a component $j$, then we have shown that it is the only optimal solution to Problem ($P_{\lambda^{T}C}$), and that $(\lambda^{T}C)_j > (\lambda^{T}C)_{j'}$ for all $j' \in J \setminus \{j\}$. 

The conditions for efficiency that were established in the preceding paragraph are not only necessary but also sufficient. Take only the case of a partially randomized $x^{\circ}$. Suppose the indices corresponding to the positive components of $x^{\circ}$ are $j_1$ through $j_p$, and suppose that there exists $\lambda \in \Lambda^{>}$ such that $(\lambda^{T}C)_{j_1} = ... = (\lambda^{T}C)_{j_p}$ and $(\lambda^{T}C)_{j_i} > (\lambda^{T}C)_{j}$ for each $i = 1, ..., p$ and $j \in J \setminus \{j_1, ..., j_p\}$. Then, from Lemma \ref{lem1}, $x^{\circ}$ is an optimal solution to Problem ($P_{\lambda^{T}C}$). Since $\lambda \in \Lambda^{>}$, it follows from Lemma \ref{lem2} that $x^{\circ} \in X_E$.

These conclusions are collected in Theorem \ref{thm1}.

\begin{theorem}
\label{thm1}
Let $x^{\circ} \in X$.
\begin{itemize}
\item[(1)] If $x^{\circ}$ is randomized, then $x^{\circ} \in X_E$ if and only if there exists $\lambda \in \Lambda^{>}$ such that $J^{*}(\lambda^{T}C) = J$.

\item[(2)] If $x^{\circ}$ is partially randomized w.r.t.\ some components $j_1, ..., j_p \in J$, then $x^{\circ} \in X_E$ if and only if there exists $\lambda \in \Lambda^{>}$ such that $J^{*}(\lambda^{T}C) = \{j_1, ..., j_p\}$.

\item[(3)] If $x^{\circ}$ is deterministic with respect to some component $j \in J$, then $x^{\circ} \in X_E$ if and only if there exists $\lambda \in \Lambda^{>}$ such that $J^{*}(\lambda^{T}C) = \{j\}$.
\end{itemize}
\end{theorem}

%\begin{proof}
%The necessity portion of each of the three statements has already been demonstrated. Sufficiency follows directly from Lemmas \ref{lem1} and \ref{lem2} and the definitions of a randomized, a partially randomized and a deterministic point.  
%\end{proof}

This characterization has two important corollaries. First, if $x^{\circ}$ is a randomized efficient solution to Problem (\ref{P}), then the theorem assures of the existence of a $\lambda \in \Lambda^{>}$ for which $J^{*}(\lambda^{T}C) = J$. By Lemma \ref{lem1}, this would mean that every $x \in X$ is an optimal solution to Problem ($P_{\lambda^{T}C}$). Per Lemma \ref{lem2}, therefore, every $x \in X$ would be an efficient solution to Problem (\ref{P}). %Conversely, if $X_E = X$, then $x^{\circ} \in X_E$.

Secondly, if $x^{\circ}$ is a partially randomized efficient solution, then the fact that there is a weight factor $\lambda \in \Lambda^{>}$ satisfying $J^{*}(\lambda^{T}C) = \{j_1, ..., j_p\}$ implies, again by virtue of Lemmas \ref{lem1} and \ref{lem2}, that every point in the set \begin{equation*}Y(j_1, ..., j_p) = \{x_{j_1}\bm{e_{j_1}} + ... + x_{j_p}\bm{e_{j_p}}: \sum_{i = 1}^{p}x_{j_i} = 1,\ x_{j_i} > 0\}\end{equation*}
belongs to $X_E$. Notice that $Y(j_1, ..., j_p)$ is the set of all points that are partially randomized w.r.t.\ $j_1, ..., j_p$, and therefore that $x^{\circ} \in Y(j_1, ..., j_p)$.

\begin{corollary}
\label{cor1}
Assume $x^{\circ} \in X$ is a randomized point. Then $x^{\circ} \in X_E$ if and only if $X_E = X$. 
\end{corollary}

\begin{corollary}
\label{cor2}
Assume $x^{\circ} \in X$ is partially randomized w.r.t.\ $j_1, ..., j_p \in J$, $p = 2, ..., n-1$. Then $x^{\circ} \in X_E$ if and only if $Y(j_1, ..., j_p) \subseteq X_E$.
\end{corollary}

\begin{remark}
\label{rem1}
From Theorem \ref{thm1} and Corollary \ref{cor1} we may draw the further conclusion that $X_E = X$ if and only if there exists $\lambda \in \Lambda^{>}$ satisfying $J^{*}(\lambda^{T}C) = J$.
\end{remark}

Although the conditions in Theorem \ref{thm1} may appear rather formidable, they provide the basis for simple computational tests for ascertaining the membership of a point $x^{\circ} \in X$ in the efficient solution set of Problem (\ref{P}). We derive these tests in the next three propositions. 

\begin{proposition}
\label{prop1}
Consider the single-objective linear program given by
\begin{equation*}
\label{$T^0$}
\tag{$T^0$}
\begin{aligned}
\max \quad & \epsilon\\
\textrm{subject to} \quad & (\lambda^{T}C)_j = (\lambda^{T}C)_{j+1},\ \forall j = 1, ..., n-1,\\
& \lambda_i \geq \epsilon, \ \forall i = 1, ..., k,\\
  &\epsilon \leq 1,\\
\end{aligned}
\end{equation*}
where $\epsilon$, $\lambda_1$, ..., $\lambda_k \in \mathbb{R}$. Then $X_E = X$ if and only if Problem {\rm (\ref{$T^0$})} has a positive optimal value $\epsilon^{*} > 0$.
\end{proposition}

\begin{proof}
As a preliminary, note that Problem (\ref{$T^0$}) is consistent (the solution where all the variables are set equal to zero is feasible) and bounded (by $1$, for example). By the fundamental theorem of linear programming, Problem (\ref{$T^0$}) has therefore an optimal value $\epsilon^{*} \in \mathbb{R}$.

We begin with the necessity portion of the proposition. Suppose $X_E = X$. By Remark \ref{rem1}, there exist $\lambda_1, ..., \lambda_{k} > 0$ such that $J^{*}(\lambda^{T}C) = J$. By the definition of $J^{*}(\lambda^{T}C)$, this means that $(\lambda^{T}C)_1 = ... = (\lambda^{T}C)_n$. Let $\lambda_{\min} = \min\{\lambda_1, ..., \lambda_k\}$, and $\epsilon = 1$. Bearing in mind that $\lambda_{\min} > 0$, set $\lambda' = \frac{\lambda}{\lambda_{\min}}$. It is then evident that $(\lambda', \epsilon)$ is a feasible solution to Problem {\rm (\ref{$T^0$})}. Furthermore, this solution achieves an objective function value $\epsilon$ equal to one. Thus, $\epsilon^{*} \geq 1 > 0$.

Conversely, suppose that $\epsilon^{*} > 0$. Choose $\lambda$ to be a $k$-dimensional vector such that $(\lambda, \epsilon^{*})$ is feasible for Problem {\rm (\ref{$T^0$})}. Then $(\lambda^{T}C)_1 = ... = (\lambda^{T}C)_n$, and $\lambda_i \geq \epsilon^{*} > 0$ for each $i = 1, ..., k$. By Remark \ref{rem1}, therefore, $X_E = X$. 
\end{proof}

\begin{proposition}
\label{prop2}
Given a point $x^{\circ} \in X$ partially randomized w.r.t.\ components $j_1, ..., j_p$, $p = 2, ..., n-1$, consider the single-objective linear program given by 
\begin{equation*}
\label{$T^1$}
\tag{$T^1(x^{\circ})$}
\begin{aligned}
\max \quad & \epsilon_m\\
\textrm{subject to} \quad & (\lambda^{T}C)_{j_i} = (\lambda^{T}C)_{j_{i+1}},\ \forall i = 1, ..., p-1, \\
& (\lambda^{T}C)_{j_1} - (\lambda^{T}C)_{j} \geq \epsilon_{j},\ \forall j \in J \setminus \{j_1, ..., j_p\},\\
& \lambda_i \geq \epsilon, \ \forall i = 1, ..., k,\\
  &\epsilon_m \leq \epsilon_{j},\ \forall j \in J \setminus \{j_1, ..., j_p\},\\
  &\epsilon_m \leq \epsilon,\\  
    &\epsilon \leq 1,\\
\end{aligned}
\end{equation*}
where $\epsilon, \epsilon_m, \epsilon_{j}, \lambda_{i} \in \mathbb{R}$, for each $j \in J \setminus \{j_1, ..., j_p\}$ and $i = 1, ..., k$. Then $x^{\circ} \in X_E$ if and only if Problem {\rm (\ref{$T^1$})} has a positive optimal value $\epsilon_{m}^{*} > 0$.
\end{proposition}

\begin{proof}
Problem (\ref{$T^1$}) admits an optimal value for the same reasons as (\ref{$T^0$}). Suppose $x^{\circ} \in X_E$. According to Theorem \ref{thm1}, there exist $\lambda_1, ..., \lambda_k > 0$ satisfying $(\lambda^{T}C)_{j_1} = ... = (\lambda^{T}C)_{j_p}$ and $(\lambda^{T}C)_{j_i} > (\lambda^{T}C)_{j}$ for each $i = 1, ..., p$ and $j \in J \setminus \{j_1, ..., j_p\}$. Let $\lambda_{\min}$ signify the smallest component in $\lambda$, and set $\lambda' = \frac{\lambda}{\lambda_{\min}}$, having noted that $\lambda_{\min} > 0$. Notice that for each $i = 1, ..., k$, $\lambda'_i \geq 1$. Let $\epsilon = 1$. Now, if for each $j \in J \setminus \{j_1, ..., j_p\}$ we let 
\begin{equation*}
\epsilon_{j} = ((\lambda')^{T}C)_{j_1} - ((\lambda')^{T}C)_{j} > 0,
\end{equation*}
then we let $\epsilon_m = \min\{\epsilon_{j}, \epsilon: j \neq j_1, ..., j_p\}$, we obtain a feasible solution to Problem (\ref{$T^1$}) with an objective function value of $\epsilon_{m} > 0$. Consequently $\epsilon_{m}^{*} > 0$. This demonstrates the \textit{only if} portion of the proposition.

To prove the \textit{if} portion, assume that $\epsilon_{m}^{*} > 0$. %Problem {\rm (\ref{$T^1$})} admits a positive optimal value $\epsilon_{m}^{*}$.
Consider any optimal solution to {\rm (\ref{$T^1$})}, say $(\lambda, \epsilon_{j}, \epsilon, \epsilon_{m}^{*})$ where $\epsilon_{j} \in \mathbb{R}^{n-p}$. The constraints of the problem dictate that \begin{equation*}\epsilon_{m}^{*} \leq \min\{\epsilon_{j}, \epsilon: j \neq j_1, ..., j_p\}.\end{equation*}
But since $\epsilon_{m}^{*}$ is positive by assumption, the above inequality implies that $\epsilon_{j} > 0$, ergo $(\lambda^{T}C)_{j_1} = ... = (\lambda^{T}C)_{j_p} > (\lambda^{T}C)_{j}$ for all $j \in J \setminus \{j_1, ..., j_p\}$, and that $\epsilon > 0$, hence $\lambda \in \Lambda^{>}$, again from the constraints of Problem (\ref{$T^1$}). In other words, $J^{*}(\lambda^{T}C) = \{j_1, ..., j_p\}$ and $\lambda \in \Lambda^{>}$. Since $x^{\circ}$ is partially randomized w.r.t.\ $j_1, ..., j_p$, it follows from Theorem \ref{thm1} that $x^{\circ} \in X_E$.
\end{proof}

\begin{proposition}
\label{prop3}
Given a point $x^{\circ}$ deterministic with respect to component $j$, consider the single-objective linear program given by
\begin{equation*}
\label{$T^2$}
\tag{$T^2(x^{\circ})$}
\begin{aligned}
\max \quad & \epsilon_m\\
\textrm{subject to} \quad & (\lambda^{T}C)_j - (\lambda^{T}C)_{j'} \geq \epsilon_{j'},\ \forall j' \in J \setminus \{j\},\\
& \lambda_i \geq \epsilon, \ \forall i = 1, ..., k,\\
  &\epsilon_m \leq \epsilon_{j'},\ \forall j' \in J \setminus \{j\},\\
  &\epsilon_m \leq \epsilon,\\  
    &\epsilon \leq 1,\\
\end{aligned}
\end{equation*}
where $\epsilon, \epsilon_m$, $\epsilon_{j'}, \lambda_i \in \mathbb{R}$, for each $j' \in J \setminus \{j\}$ and $i = 1, ..., k$. Then $x^{0} \in X_E$ if and only if Problem {\rm (\ref{$T^2$})} has a positive optimal value $\epsilon_{m}^{*} > 0$.
\end{proposition}

\begin{proof}
The proof is analogous to that of Proposition \ref{prop2}. In the selection of a suitable feasible solution for the necessity portion of the proposition, construct $\lambda'$ in the same manner as above, set $\epsilon = 1$, $\epsilon_{j'} = ((\lambda')^{T}C)_{j} - ((\lambda')^{T}C)_{j'}$ for each $j' \in J \setminus \{j\}$, and choose $\epsilon_m = \min\{\epsilon_{j'}, \epsilon: j' \neq j\}$. The treatment of the sufficiency portion may proceed in exactly the fashion of Proposition \ref{prop2}.
\end{proof}

A procedure for determining whether a point $x^{\circ} \in X$ lies in $X_E$ follows at once from these propositions. We begin by solving Problem (\ref{$T^0$}). If the optimal value is positive, the procedure stops: not only is $x^{\circ}$ an efficient solution of Problem (\ref{P}), but so is the rest of $X$. If not, we ask whether $x^{\circ}$ is randomized. If so, then we have established that $x^{\circ}$ is dominated. If not, we solve Problem (\ref{$T^1$}) or (\ref{$T^2$}), depending on whether $x^{\circ}$ is partially randomized or deterministic. In either case, should the optimal value be positive, we conclude that $x^{\circ}$ is efficient; otherwise we rule it out as dominated.

There are several important practical points to keep in mind. First of all, if a candidate $x^{\circ}$ is partially randomized with respect to, say, $j_1, ..., j_p$, then the outcome of the procedure as it pertains to the status of $x^{\circ}$ is, by Corollary \ref{cor2}, valid for \textit{all} points that are partially randomized w.r.t.\ $j_1, ..., j_p$. This makes it redundant to execute the procedure on a partially randomized point if a point partially randomized w.r.t.\ the same components has already been investigated.

%and the procedure declared it efficient, then we have shown, in effect, that every other point that is partially randomized with respect to those components is also an efficient solution (see Corollary \ref{cor2}).  

%There appear to be no technical difficulties attending the implementation of this procedure. As noted, Problems (\ref{$T^0$}), (\ref{$T^1$}) and (\ref{$T^2$}) are standard linear programming problems that can readily be solved by the simplex method. 

Secondly, as the example of Section \ref{sec3} will highlight, the procedure may be incorporated into a scheme for the generation of $X_E$ or a subset thereof. Obviously, by examining various candidates $x^{\circ}$ for efficiency, we may generate multiple efficient solutions. In particular, by keeping strictly to deterministic points, of which there is a finite number, it should be possible to locate all deterministic efficient solutions. To accomplish this, one would need only solve Problem (\ref{$T^2$}) for each $x^{\circ} = \bm{e_j}$, $j = 1, ..., n$, and record those indices $j$ corresponding to a positive optimal value. This search is guaranteed to return at least one efficient solution, because in a multiple objective linear program such as Problem (\ref{P}), at least one extreme point of the feasible polyhedron corresponds to an efficient solution \citep[Corollary 1.5]{evans1973revised}. It is a straightforward exercise to show that the deterministic points represent the extreme points of $X$.

Finally, the implementation of the procedure is uncomplicated given that the programs to be solved are standard linear programs.    

%One striking aspect of this procedure is that it does not call for computations other than those involved in the solution of ordinary linear programming problems, which can readily be accomplished using the simplex method.

We have seen that one useful function of Problem (\ref{$T^0$}) is indicating whether $X_E = X$. The next example shows that this pathological situation arises even in nontrivial instances of Problem (\ref{P}).  %Under what conditions on the criteria of Problem (\ref{P}) is $X_E = X$? An interesting prior question is whether one can construct nontrivial instances of (\ref{P}) in which this pathological situation arises. Example \ref{ex1} provides such an instance.

\begin{example}
\label{ex1}
In Problem (\ref{P}), let $k = n = 3$ and
\[C = \begin{pmatrix}
1 & 2 & -5\\
2 & 1 & -1 \\
-3 & -2 & 9
\end{pmatrix}.\]
%so that for all $x \in X$, $c_1^{T}x = x_1 + 2x_2 - 5x_3,\ c_2^{T}x = 2x_1 + x_2 - x_3$, and $c_3^{T}x = -3x_1 -2x_2 + 9x_3$. 
Notice that \begin{equation*}\begin{pmatrix} 1 \\ 2 \\ 1 \end{pmatrix}^{T}C = (2, 2, 2),\end{equation*}
so that for $\lambda = (1, 2, 1)^{T} \in \Lambda^{>}$, \begin{equation*}(\lambda^{T}C)_1 = (\lambda^{T}C)_2 = (\lambda^{T}C)_3 = 2.\end{equation*}
%\begin{equation*}c_{11} + 2c_{21} + c_{31} = c_{12} + 2c_{22} + c_{32} = c_{13} + 2c_{23} + c_{33} = 2,\end{equation*}
%$\lambda = (1, 2, 1) > 0$, \begin{equation*}(\lambda^{T}C)_1 = (\lambda^{T}C)_2 = (\lambda^{T}C)_3 = 2.\end{equation*}

%for all $x \in X$, \begin{equation*}(c_1^{T} + 2c_2^{T} + c_3^{T})x = 2(x_1 + x_2 + x_3) = 2.\end{equation*}
It follows from Remark \ref{rem1} that $X_E = X$. 

\end{example}

Theorem \ref{thm2} points out conditions under which $X_E = X$ for a bicriterion Problem (\ref{P}). Remark \ref{rem2} delineates a class of problems for which these conditions are both sufficient and necessary. Before stating the theorem, we prove a useful lemma that is readily obtained from our proof of Proposition \ref{prop1}.

\begin{lemma}
\label{lem3}
Let $\epsilon^{*}$ represent the optimal value of Problem {\rm (\ref{$T^0$})}. Then $\epsilon^{*} > 0$ if and only if $\epsilon^{*} = 1$.
\end{lemma}

\begin{proof}
Suppose $\epsilon^{*} > 0$. Then $\lambda'$ as constructed in the first half of the proof of Proposition \ref{prop1} is such that $(\lambda', \epsilon = 1)$ is a feasible solution to Problem (\ref{$T^0$}). Thus, $\epsilon^{*} \geq 1$. Since Problem (\ref{$T^0$}) constrains $\epsilon^{*}$ to be bounded above by $1$, it follows that $\epsilon^{*} = 1$. 
\end{proof}

\begin{theorem}
\label{thm2}
Assume that $k = 2$ and that $c_{1, j} \neq c_{1, j+1}$ for each $j = 1, ..., n-1$. If

\begin{equation}\label{eq_1_thm2}\frac{c_{2, 2} - c_{2, 1}}{c_{1, 1} - c_{1, 2}} = \frac{c_{2, 3} - c_{2, 2}}{c_{1, 2} - c_{1, 3}} = \dots = \frac{c_{2, n} - c_{2, n-1}}{c_{1, n-1} - c_{1, n}} > 0,\end{equation}
then $X_E = X$.
%\begin{equation}\tag{C2}\label{eq_2_thm2}c_{2, j} \neq c_{2, j+1} \forall j = 1, ..., n-1,\ \frac{c_{1, 1} - c_{1, 2}}{c_{2, 2} - c_{2, 1}} = \frac{c_{1, 2} - c_{1, 3}}{c_{2, 3} - c_{2, 2}} = \dots = \frac{c_{1, n-1} - c_{1, n}}{c_{2, n} - c_{2, n-1}} > 0.\end{equation}

\end{theorem}

\begin{proof}
Suppose (\ref{eq_1_thm2}) holds. From Lemma \ref{lem3} and Proposition \ref{prop1}, $X_E = X$ if and only if the optimal value of Problem (\ref{$T^0$}) equals one. To prove this theorem, therefore, one need only construct a feasible solution to Problem (\ref{$T^0$}) having an objective function value of one. To this end, let, for example, $\lambda_{2} = 1$ and \begin{equation*}\lambda_{1} = \lambda_{2}\cdot\frac{c_{2, 2} - c_{2, 1}}{c_{1, 1} - c_{1, 2}}.\end{equation*}
Let $\lambda_{\min} = \min\{\lambda_{1}, \lambda_{2}\}$. Since by assumption $\frac{c_{2, 2} - c_{2, 1}}{c_{1, 1} - c_{1, 2}} > 0$, $\lambda_{\min} > 0$. Set $\lambda'_{1} = \frac{\lambda_{1}}{\lambda_{\min}}$ and $\lambda'_{2} = \frac{\lambda_{2}}{\lambda_{\min}}$. Then $\lambda'_{1}, \lambda'_{2} \geq 1$. 

Under our hypotheses, $\lambda_{1}$ and $\lambda_{2}$ satisfy the system of equations given by \begin{equation*}\lambda_{1}(c_{1, j} - c_{1, j+1}) = \lambda_{2}(c_{2, j+1} - c_{2, j}),\ \forall j = 1, ..., n-1,\end{equation*}
hence
\begin{equation*}\lambda'_{1}(c_{1, j} - c_{1, j+1}) = \lambda'_{2}(c_{2, j+1} - c_{2, j}),\ \forall j = 1, ..., n-1,\end{equation*}
a system that may be rewritten as %\begin{equation*}\lambda'_{1}c_{1, j} + \lambda'_{2}c_{2, j} = \lambda'_{1}c_{1, j+1} + \lambda'_{2}c_{2, j+1},\ \forall j = 1, ..., n-1,\end{equation*}
%or, by setting $\lambda' = (\lambda'_1, \lambda'_2)$, 
\begin{equation*}(\lambda'^{T}C)_j = (\lambda'^{T}C)_{j+1},\ \forall j = 1, ..., n-1,\end{equation*}
where $\lambda' = (\lambda'_1, \lambda'_2)$. This, coupled with the fact that $\lambda'_1, \lambda'_2 \geq 1$, proves that $(\lambda', \epsilon = 1)$ is feasible for Problem (\ref{$T^0$}), whence $X_E = X$.
%Now, since $\lambda_{2} = 1$ and $\frac{c_{2, 2} - c_{2, 1}}{c_{1, 1} - c_{1, 2}} \geq 1$ by assumption, $\lambda_{1} \geq 1$. 

%In case (\ref{eq_2_thm2}) holds, repeat the same argument as before with $\lambda_1 = 1$ and $\lambda_{2} = \lambda_{1}\cdot\frac{c_{1, 1} - c_{1, 2}}{c_{2, 2} - c_{2, 1}}$. The proof is complete. 
\end{proof}

%\begin{remark}
%An inspection of the last proof shows that Theorem \ref{thm2} would still hold if, instead of requiring $c_{1, j}$ to be distinct from $c_{1, j+1}$ for each $j = 1, ..., n-1$, we imposed the analogous requirement on criterion $2$, provided $c_1$ and $c_2$ were interchanged in (\ref{eq_thm2}).
%\end{remark}

\begin{remark}
\label{rem2}
It may be illuminating to interpret Theorem \ref{thm2} within the context of certain bicriterion problems in two dimensions. When $n = 2$, notice that $X = \{(x, 1-x): x \geq 0\}$. Let $f_1(x)$ and $f_2(x)$ denote the criteria $c_1^{T}x$ and $c_{2}^{T}x$, respectively, so that \begin{equation}\label{eq_f1}f_1(x) = (c_{1,1} - c_{1,2})x + c_{1,2}\end{equation} and \begin{equation}\label{eq_f2}f_2(x) = (c_{2,1} - c_{2,2})x + c_{2,2}\end{equation}
for all $x \geq 0$. Suppose that $f_1(x)$ is monotonically increasing and that $f_2(x)$ is monotonically decreasing. Then it is clear that for any point $(x, 1-x) \in X$, an improvement in the value of $f_1(x)$ can only come at the expense of a reduction in the value of $f_2(x)$, which by definition means that $(x, 1-x)$ is an efficient solution to Problem (\ref{P}). The identical conclusion could have also been arrived at by invoking Theorem \ref{thm2}. Indeed, in view of (\ref{eq_f1}) and (\ref{eq_f2}) and of the monotonicity of $f_1(x)$ and $f_2(x)$, we have assumed that $c_{1,1} - c_{1,2} > 0$ and $c_{2,1} - c_{2,2} < 0$, hence \begin{equation*}\frac{c_{2, 2} - c_{2, 1}}{c_{1, 1} - c_{1, 2}} > 0,\end{equation*} precisely the condition required by Theorem \ref{thm2}.

Where strictly monotonic criteria are involved, the assumption of opposite monotonicity (or, equivalently, condition (\ref{eq_1_thm2}) in Theorem \ref{thm2}) is necessary for $X_E$ to be equal to $X$. This is due to the fact that if $f_1(x)$ and $f_2(x)$ are both increasing or both decreasing, then they must attain their maxima at the same point $x^{\circ} \geq 0$ ($x^{\circ} = 1$ in the former case, $x^{\circ} = 0$ in the latter one), from which it follows that $X_E = \{(x^{\circ}, 1-x^{\circ})\}$. 

\end{remark}
% considerations of efficiency

\begin{remark}
\label{rem3}
Implicit in much of the development that succeeded Lemma \ref{lem2} has been the assumption that $n \geq 3$. This is simply a reflection of the fact that when $n = 2$, Lemma \ref{lem1}, the other important lemma in the development, does not apply in its original form, but rather in the modified form of Remark \ref{rem0}. In particular, when $n = 2$, the optimal solution set of Problem ($P_{\lambda^{T}C}$) is, for any $\lambda \in \mathbb{R}^{k}$, either $X$ or $\{\bm{e_j}\}$ for some $j = 1, 2$. In this case one can only speak of deterministic and randomized points. A close inspection of the proofs of results concerning these two classes of points shows that they, in addition to the relevant parts of the procedure, would still be valid if $n = 2$.       
\end{remark}

\section{Example}\label{sec3}

In order to illustrate the procedure described in Section \ref{sec2}, we shall consider an instance of Problem (\ref{P}) where $n = k = 3$ and the criteria matrix is given by \[C = \begin{pmatrix}
1 & 2 & -4\\
2 & -5 & 1 \\
0 & 3 & -\frac{1}{2}
\end{pmatrix}.\]
We shall see how the procedure can be used not merely to detect individual efficient solutions but to exhaustively enumerate them.

Before testing specific points for efficiency, the procedure dictates that we solve Problem (\ref{$T^0$}) to ascertain whether or not $X_E = X$. The simplex method yields an optimal value of zero. Thus, $X_E \neq X$, and so the procedure continues.

Let us consider the arbitrary point $x^{\circ} = (0.55, 0.45, 0) \in X$. To determine the status of $x^{\circ}$, a partially randomized point, we must solve Problem (\ref{$T^1$}). The optimal value of this problem equals one. Consequently, $x^{\circ} \in X_E$, and every partially randomized point with respect to $j_1 = 1$ and $j_2 = 2$ also lies in $X_E$.

Next, we turn to $x^{\circ} = (0.3, 0, 0.7)$. The corresponding problem is, as with the previous point, Problem (\ref{$T^1$}). Solution of (\ref{$T^1$}) yields an optimal value of zero. As a result, $x^{\circ} \notin X_E$. Furthermore, there exist no efficient solutions among the points that are partially randomized w.r.t.\ $j_1 = 1$ and $j_2 = 3$.

To complete the investigation of partially randomized points, choose, for example, $x^{\circ} = (0, 0.85, 0.15)$. The optimal value of Problem (\ref{$T^1$}) is nil. This implies that $x^{\circ} \notin X_E$, and that no point that is partially randomized w.r.t.\ $j_1 = 2$ and $j_2 = 3$ can be an efficient solution.

From Section \ref{sec2}, we know that at least one extreme point of $X$, and therefore one deterministic point, must be an efficient solution to Problem (\ref{P}). The deterministic points in this example are $(1, 0, 0), (0, 1, 0)$ and $(0, 0, 1)$. To determine their statuses, Problem (\ref{$T^2$}) must be solved for each of the three. Upon doing so, we find an optimal value of one for $x^{\circ} = (1, 0, 0)$ and $x^{\circ} = (0, 1, 0)$, and an optimal value of zero for $x^{\circ} = (0, 0, 1)$. Consequently, $(1, 0, 0), (0, 1, 0) \in X_E$ and $(0, 0, 1) \notin X_E$.

To summarize, we have demonstrated that the efficient solution set is given by \begin{flalign*}X_E &= Y(1, 2) \cup \{(1, 0, 0), (0, 1, 0)\}\\
&= \{(x_1, x_2, 0): x_1 + x_2 = 1;\ x_j \geq 0, j = 1, 2\}.\end{flalign*}
%The criteria to be optimized over $X = \{(x_1, x_2, x_3) \in \mathbb{R}^{3}: x_1 + x_2 + x_3 = 1,\ x_1, x_2, x_3 \geq 0\}$ are therefore \begin{flalign*}c_{1}^{T}x &= x_1 + 2x_{2} -4x_{3},\\
%c_{2}^{T}x &= 2x_1 - 5x_{2} + x_{3},\\
%c_{3}^{T}x &= 3x_{2} - \frac{x_{3}}{2}.\end{flalign*}
%\begin{equation*}c_{2}^{T}x = 2x_1 - 5x_{2} + x_{3},\end{equation*}
%\begin{equation*}c_{3}^{T}x = 3x_{2} - \frac{x_{3}}{2},\end{equation*}

\section{Summary}
\label{sec4}
This paper gave a characterization of an efficient solution for the multiple objective programming problem (\ref{P}). The characterization, which distinguishes between randomized, partially randomized and deterministic solutions, aided in establishing two properties of Problem (\ref{P}). First, it was shown that unless the entire feasible region is efficient, no randomized efficient solution can exist. Secondly, it was found that the status of a partially randomized point is a function of the positions of the positive components rather than of their values. Based on these and other results, a computational procedure was presented that ascertains whether a given feasible point is an efficient solution. For each point considered, the procedure solves one of two linear programs, and, if the optimal value is positive, declares the point efficient. The procedure does not require, as is traditionally the case in multiple objective linear programming, that candidates for efficiency be extreme points. A simple illustration of the procedure was provided.

%The assumptions of the previous section subsume Yu's framework, and his result is a direct corollary of our theorem. For $G = \mathbb{R}^{q}$, $A = Y \subseteq \mathbb{R}^{q}$, $B = \Lambda^{\leqq} \subseteq \mathbb{R}^{q}$ and $R = \geqq$, we have that $(\mathbb{R}^{q}, +)$ is a group, $\geqq$ has the desired properties, $0 \in \Lambda^{\leqq}$ and $0 \geqq d$ for all $d \in \Lambda^{\leqq}$, so that $\mathscr{E}(Y + \Lambda^{\leqq}) = \mathscr{E}(Y)$ by virtue of the theorem.

%In the context of multiple objective programming, where $G = \mathbb{R}^{q}$ and $R$ is usually taken to be $\geqq$ \citep{soland1979multicriteria}, it can be shown that any pair of non-empty sets $(A, B)$ meets the theorem's requirements provided $B$ has the form $B = \{d \in \mathbb{R}^{q}: -d \geqq s\} \cup \{0\}$ for some $s \geqq 0$. In general it would appear that, except when $A+B$ equals $A$ (this arises, for instance, when $B = A$ and $A$ is a convex cone), the conditions $0_{G} \in B$ and $0_{G} R b$ for all $b \in B$, or some equivalent conditions, are necessary ones. Further research will have to be undertaken to determine (or refute) this formally.

\backmatter

%\bmhead{Acknowledgements}

%Acknowledgements are not compulsory. Where included they should be brief. Grant or contribution numbers may be acknowledged.

%Please refer to Journal-level guidance for any specific requirements.

\section*{Declarations of interest}
The author has no competing interests to declare.

\bibliography{sn-bibliography}% common bib file
%% if required, the content of .bbl file can be included here once bbl is generated
%%\input sn-article.bbl

\end{document}